\documentclass{amsart}
\usepackage{amsmath}
\usepackage{amssymb}
\usepackage{amsfonts}
\usepackage{amsthm}

\usepackage{soul}
\usepackage{tikz}				

\everymath{\displaystyle}

\theoremstyle{definition}
\newtheorem{theorem}{Theorem}
\newtheorem{corollary}{Corollary}
\newtheorem*{theorem*}{Theorem}
\newtheorem{lemma}[theorem]{Lemma}
\newtheorem*{definition}{Definition}

\theoremstyle{remark}

\newcommand{\N}{\mathbb{N}}

\newcommand{\Orb}{\mathrm{Orb}}

\title[The specification property]{The specification property on a set-valued map and its inverse limit}

\author[B. E. Raines]{Brian E. Raines}
\address[B. E. Raines]{Department of Mathematics, Baylor University, Waco, TX 76798--7328,USA}
\email{brian\_raines@baylor.edu}

\author[T. Tennant]{Tim Tennant}
\address[T. Tennant]{Department of Mathematics, Baylor University, Waco, TX 76798--7328,USA}
\email{timothy\_tenannt@baylor.edu}


\newcommand{\invlim}{\lim\limits_{\longleftarrow}}

\begin{document}
\maketitle
\begin{abstract}
In this paper we consider dynamical properties of set-valued mappings and their implications on the associated inverse limit space.  Specifically, we define the specification property and topological entropy for set-valued functions and prove some elementary results of these properties.  We end with a few results regarding invariant measures for set-valued functions and their associated inverse limits.
%

\end{abstract}

\section{Introduction}

In this paper we consider the dynamics of set-valued functions and their inverse limit spaces.  Let $X$ be a compact metric space and let $2^X$ be the hyperspace of nonempty closed subsets of $X$.  Let $F:X\to 2^X$ be an upper semi-continuous function. We call $F$ a \emph{set-valued mapping}. The \emph{inverse limit space induced by $F$} is the space $$\invlim F=\{(x_0,x_1,\dots)\in X^\N: x_{i+1}\in F(x_i)\}$$ considered as a subspace of the Tychonoff product $X^\N$.  Associated with the inverse limit spaces is a \emph{shift map} $$\sigma(x_0,x_1,\dots)=(x_1,x_2\dots).$$  This is a continuous well-defined function which mimics the dynamics of $F$ on $X$.  Of course, the trade-off for well-definedness is the inherently more complicated topology of the domain space, $\invlim F$.   These mappings and their inverse limits have arisen in several applications of topological dynamics to economics, \cite{kennedystockmanyorke} \& \cite{RainesStockman}.  Much has been written just on the topological structure of such inverse limit spaces, \cite{Ingram}.

We extend the definitions of the specification property and topological entropy from the usual single-valued function setting to the setting of set-valued mappings.  We prove that specification implies topological mixing and positive entropy.  We also extend the notion of shadowing to the set-valued mapping case.  We show that if $X$ is compact and connected and $F$ has shadowing and a dense set of periodic points then it also has a slightly weaker version of the specification property (as is known in the single-valued setting \cite{araichinen}.)  We then give a few results on the dynamics induced on $\sigma$, specifically, we show that if $F$ has the specification property then so does $\sigma$.  We end the paper with a few results regarding invariant measures for set-valued mappings. This is closely related to results from \cite{kennedyrainesstockman}.

\section{Preliminaries.}
Let $X$ be a compact metric space, and let $2^X$ be the hyperspace of closed subspaces of $X$ with the Hausdorff metric.  Let $f:X\to X$ be a continuous function, and let $F\colon  X\to 2^X$ be a set-valued mapping.  Recall that the sequence $(x,f(x),f^2(x),\ldots)$ is the \emph{orbit} of $x$ under $f$. If there exists $n\in\N$ such that $f^n(x) = x$, we call $x$ \emph{periodic}, with period $n$.

We begin with a few simple extensions of definitions from the single-valued case.  Notice that since $F$ is a set-valued mapping, orbits of $F$ are no longer uniquely determined by their initial condition.

\begin{definition}
An \emph{orbit} of a point $x\in X$ is a sequence $(x_i)_{i=0}^\infty$ such that $x_{i+1}\in F(x_{i})$ and $x_0 = x$.
\end{definition}
\begin{definition}
Let $x\in X$.  Let $(x_i)_{i=0}^\infty$ be an orbit of $x$.  The orbit  is said to be \emph{periodic} if there is some $n\in \N$ such that $(x_i)_{i=0}^\infty = (x_j)_{j=n}^\infty$.  The smallest such $n$ is called the \emph{period} of the orbit.
\end{definition}
Note that it is not necessarily the case that if there is some $j\in \N$ such that $x_0 = x_j$, then the orbit $(x_j)_{j=0}^\infty$ is periodic.
\begin{definition}
  The \emph{complete orbit} of a point $x$, denoted $CO(x)$, is the collection of all orbits of $x$.
\end{definition}
\begin{definition}
The \emph{set of complete orbits} of a set-valued mapping $F$, denoted $CO(F)$, is the collection of complete orbits of each point of $X$, i.e. $$CO(F)=\{CO(x):x\in X\}.$$
\end{definition}
\begin{definition}
We say that $F$ has the \emph{specification property} if, for any $\epsilon > 0$, there exists $M\in \N$ dependent only on $\epsilon$ such that, for any $x^1,\ldots, x^n\in X$, any $a_1\leq b_1 < \ldots < a_n \leq b_n$ with $a_{i+1}-b_i > M$, and any orbits $(x^i_j)_{j=0}^\infty$, and for any $P > M+b_n-a_1$, there exists a point $z$ that has an orbit $(z_j)_{j=0}^\infty$ such that $d(z_j,x^i_j) < \epsilon$ for $i\in \{1,\ldots, n\}$ and $a_i \leq j \leq b_i$, and $z_P = z$.  \end{definition}
Note that this definition requires $F$ to be surjective, in the sense that for any $y\in X$, there exists $x\in X$ such that $y\in F(x)$.
  \begin{definition}
  We say that $F$ is \emph{topologically mixing} if, for any non-empty open $U$ and $V$ in $X$, there is an $M\in \N$ such that for any $m > M$ there is an $x\in U$ with an orbit $(x_j)_{j=0}^\infty$ with $x_m\in V$ .
  \end{definition}

The notion of topological entropy has been studied extensively since it was introduced in 1965.  Positive entropy is a strong indicator of topological chaos.
\begin{definition}
 Let $x\in X$, and let  $\Orb_n(F, x) =  \{(x_1,\ldots, x_n): x_{i+1}\in F(x_i)\text{ and }x_1=x\}$.  Call $\Orb_n(F, x)$ the \emph{partial orbits of $x$ under $F$ of length $n$.}  Let $$\Orb_n(F)=\bigcup_{x\in X} \Orb_n(F,x).$$
\end{definition}
  \begin{definition}
  We say that $A\subset \Orb_n(F)$ is \emph{$(n,\epsilon)$-spanning} if for any $(x_1,\dots x_n)\in \Orb_n(F)$, there is a $(y_1,\dots y_n)\in A$ such that $d(x_i,y_i)< \epsilon$ for all $i\in \{1,\ldots, n\}$.  Let $r_n(\epsilon)$ denote the minimum cardinality of an $(n,\epsilon)$-spanning set.
 \end{definition}
 \begin{definition} We define the \emph{topological entropy} of $F$ as
\[h(F) = \lim_{\epsilon \rightarrow 0}\limsup_{n\rightarrow \infty} \frac{1}{n}\log (r_n(\epsilon)).\]
\end{definition}
\begin{definition}We say that $A\subset \Orb_n(F)$ is \emph{$(n,\epsilon)$-separated} if for any $$(x_1,\ldots,x_n),(y_1,\ldots,y_n)\in A,$$ $d(x_i,y_i)\geq \epsilon$ for at least one $i\in \{1,\ldots, n\}$.  Let $s_n(\epsilon)$ denote the greatest cardinality of an $(n,\epsilon)$-separated set.
\end{definition}
It is possible to define topological entropy for a set-valued function using $s_n(\epsilon)$ instead of $r_n(\epsilon)$.  In that case the topological entropy of $F$ is given by \[h(F)=\lim_{\epsilon \rightarrow 0}\limsup_{n\rightarrow \infty} \frac{1}{n}\log (s_n(\epsilon)).\]   The following results show that these definitions are indeed equivalent. While immediate, they are included for completeness. 

\begin{lemma}
$r_n(\epsilon) \leq s_n(\epsilon) \leq r_n(\frac{\epsilon}{2})$.
\end{lemma}
\begin{theorem}
For $F\colon  X\rightarrow 2^X$, the definitions of topological entropy using $(n,\epsilon)$-separated sets and $(n,\epsilon)$-spanning sets are equivalent.
\end{theorem}

\section{Specification.}
We begin by considering the implications of the specification property for a set-valued mapping, $F$.
 \begin{theorem}
 If $F\colon  X\to 2^X$ is a set-valued mapping that has the specification property, then $F$ has topological mixing.
 \end{theorem}
 \begin{proof}
 Let $U$ and $V$ be non-empty open sets.  Let $x\in U$, $y\in V$. Let $\epsilon > 0$ be chosen such that $B_\epsilon(x) \subset U$ and $B_\epsilon(y) \subset V$.  Let $M\in \N$ witness the specification property for this $\epsilon$.  Let $y^1\in X$ be chosen such that there is an orbit $(y^1_j)_0^\infty$ with $y^1_{M+1} = y$. Let  $(x_i)_0^\infty$ be any orbit of $x$. Let $a_1=b_1 = 0$, and $a_2=b_2=M+1$, then there is a point $z^1$ that has an orbit $(z^1_i)_0^\infty$, such that
\[d(z^1,x) < \epsilon\;\;\;\; \text{and}\;\;\;\; d(z^1_{M+1}, y^1_{M+1}) < \epsilon.\]  Thus $z^1\in U$, and $z^1_{M+1}\in V$. Now let $y^m\in X$ such that there is an orbit $(y^m_j)_{j=0}^\infty$ with $y^m_{M+m} = y$.  Then for $a_1 = b_1 = 0$, and $a_2 = b_2 = M+m$, there is a point $z^m$ with an orbit $(z^m_j)_0^\infty$ such that \[d(z^m,x) < \epsilon \;\;\;\;\text{and}\;\;\;\; d(z^m_{M+m}, y^m_{M+m}) < \epsilon.\]  Then $z^m\in U$, and $z^m_{M+m} \in V$.
 \end{proof}
 The following lemma is immediate, but is used in Theorem 5 and so we include it for completeness.
\begin{lemma} For $\epsilon_1 < \epsilon_2$, $s_n(\epsilon_1) \leq s_n(\epsilon_2)$.
\end{lemma}

\begin{theorem}

If $F\colon X\to 2^X$ is a set-valued mapping that has the specification property, then $F$ has positive entropy.
\end{theorem}
\begin{proof}
Let $x, y \in X$.  Let $\epsilon > 0$ such that $d(x,y) > 3\epsilon$.  Let $M$ witness the specification property for this $\epsilon$.  Let $z^1,\ldots, z^n$ be chosen such that $z^i\in \{x,y\}$ for some fixed $n$.  Let $(z^i_j)_{j=0}^\infty$ be an orbit of each $z^i$.  Let $a_i = b_i = (i-1)M$.  Then there is some $z$ that has an orbit $(z_j)_{j=0}^\infty$ such that $d(z_{(i-1)M}, z^i_{(i-1)M}) < \epsilon$.
Let $\hat{z}^1,\ldots, \hat{z}^n$ be a different such choice of $x$'s and $y$'s.
Then for orbits $(\hat{z}^i_j)_{j=0}^\infty$ there is some $\hat{z}$ (possibly equal to $z$) that has an orbit $(\hat{z}_j)_{j=0}^\infty$ that follows the $\hat{z}^i$s.  Now, there is some $i\in \{1,\ldots, n\}$ such that $\hat{z}^i \neq z^i$.  Then  $d(z_{(i-1)M}, \hat{z}_{(i-1)M}) > \epsilon$, for that fixed $i$.  Thus there are $2^n$ many $(nM,\epsilon)$-separated orbits. Then for a fixed $\epsilon$,
\[\limsup_{n\rightarrow \infty} \frac{1}{n} \log (s_n(\epsilon)) \geq \limsup_{n\rightarrow\infty} \frac{1}{nM} \log 2^{n} = \frac{1}{M}\log 2.\]
Thus by Lemma $4$, $F$ has positive entropy.
\end{proof}

Whenever a chaotic property such as specification is introduced, a natural question is to determine sufficient conditions for it to appear.  The following theorem gives a sufficient condition to get a slightly weaker property than specification. First, we introduce the notion of shadowing.
\begin{definition} A sequence $(x_i)_{i=1}^\infty$ is called a \emph{$\delta$-pseudo-orbit} if $d(F(x_i),x_{i+1}) < \delta$ for all $i\in \N$.
\end{definition}
\begin{definition} Let $F\colon  X\rightarrow 2^X$ be a set-valued map. We say $F$ has \emph{shadowing} if, for any $\epsilon > 0$, there exists $\delta > 0$ such that for any $\delta$-pseudo-orbit $(x_i)_{0}^\infty$, there  exists a point $z\in X$ with an orbit $(z_i)_0^\infty$ such that $d(z_i,x_i)< \epsilon$ for all $i\in \N$.
\end{definition}
\begin{theorem}
Suppose that $F$ is a set-valued map on a continuum $X$ that has shadowing and a dense set of points that each have at least one periodic orbit.  Then for any $\epsilon > 0$, there is an $M\in \N$ such that for any $x^1,\ldots, x^n \in X$, any $a_1 \leq b_1 < a_2 \leq b_2 < \ldots <a_n \leq b_n$, any orbits $(x^i_j)_{j=0}^\infty$, there is a point $z\in X$ with a orbit $(z_j)_{j=0}^\infty$ such that $d(z_j, x^i_j) < \epsilon$ for $1\leq i \leq n$, $a_i \leq j \leq b_i$.
\end{theorem}
\begin{proof}

  Let $\epsilon > 0$.  Let $\delta$ witness shadowing for this $\epsilon$.  Consider the open cover 
$$\mathcal{B} =\{B(\delta/2,q): q \text{ has at least one periodic orbit}\}.$$ Let $\mathcal{A} =\{B(\delta/2,q^1),\ldots, B(\delta/2,q^m)\}$ be a finite subcover of $\mathcal{B}$.  For each $1\leq i \leq m$, choose a periodic orbit of $q^i$ and denote the period of that orbit by $p_i$.
Let $M_i$ be the collection of sums of $i$ many elements from the list of $\{p_1,\ldots, p_m \}$, i.e.,
\[M_i = \{p_{j_1}+p_{j_2}+\ldots+p_{j_i}: p_{j_k} \in \{p_1,\ldots, p_m\},\text{ for } 1\leq k \leq i\}.\]   Let $M_0 = \prod_1^{m} M_i$, and $M = 2M_0$.  This will be the $M$ witnessing specification for the given $\epsilon$.

Let $x^1,\ldots, x^n$ be points of $X$.  Let $a_1\leq b_1<a_2\leq\ldots \leq b_n$, with $a_{i+1}-b_n >M$.  Let $(x^i_j)_{j=0}^\infty$ be orbits of these points. \\
We build a pseudo-orbit that follows the orbits of the points $x^i$, and then invoke shadowing.
 Let $r_i$ and $c_i'$  be natural numbers such that $a_{i+1}-b_i = c_i'M_0 + r_i$, with $0 \leq r_i < M_0$, for $1\leq i \leq n$.  Note that as $a_{i+1}-b_i > M$, $c_i'$ is at least 2.

 It is important to note that for any two points $u,v$ in $X$, there is a finite sequence $s_1,\ldots, s_r$ with $r\leq m$, where $m$ is the cardinality of $\mathcal{A}$, given above, such that
 \[d(s_i,s_{i+1}) < \delta,\;\;\;\; d(u,s_1) < \delta,\;\;\;\; d(s_r, v) < \delta,\] and $s_i\in \{q_1,\ldots, q_m\}$ for each $1\leq i \leq r$.
 So for each pair $x^i_{b_{i}+r_i}$, $x^{i+1}_{a_{i+1}}$, choose a finite sequence $s^i_1,\ldots, s^i_{\eta_i}$ such that
 \[d(s^i_j,s^i_{j+1}) < \delta,\;\;\;\; d(x^i_{b_i+r_i},s^i_1) < \delta,\;\;\;\; d(s^i_{\eta_i}, x^{i+1}_{a_{i+1}}) < \delta.\]
   Then each $s^i_j$ has a periodic orbit of length $p^i_j$, with $p^i_j = p_n$, for some $1 \leq n \leq m$.
   Let $(s^i_{k,j})_{j=0}^\infty$ be that periodic orbit, and so we have \[s^i_{k,p^i_k} = s^i_{k,0} = s^i_k,\]
   for each $1\leq i \leq n$, and $1\leq j \leq \eta_i$.

  Let $\displaystyle{c_i = \frac{c_i'M_0}{\sum_{k=0}^{\eta_i} p^i_{k}}}.$  Now we are ready to define our pseudo-orbit.  Let $(w_k)_{k=0}^\infty$ be defined such that
  \[w_k = \begin{cases}
    x^1_k 		& 		0 \leq k < b_1+r_1\\
	s^1_{1,k} 	&		b_1+r_1 \leq k < b_1 + r_1 + c_1p^1_1\\
	s^1_{2,k}	&		b_1+r_1+c_1p^1_1 \leq k < b_1+r_1 +c_1(p^1_1+p^1_2)\\
	\vdots		&		\hspace{1in}\vdots		\\
	s^1_{\eta_1,k} & 	b_1+r_1+c_1\sum_{j=0}^{\eta_1-1} p^1_j \leq k < b_1+r_1+c_1\sum_{j=0}^{\eta_1} p^1_j\\
	x^2_k		&		a_2 \leq k < b_2 + r_2\\
	s^2_{1,k}	&		b_2+r_2 \leq k < b_2+r_2+c_2p^2_1\\
	\vdots		&		\hspace{1in}\vdots\\
	x^n_k		&		a_n \leq k < b_n+r_n\\
	\vdots 		&		\hspace{1in}\vdots\\
	s^n_{\eta_n,k} &	b_n+r_n+c_n\sum_{j=0}^{\eta_n-1} p^n_j \leq k < b_n+r_n+c_n\sum_{j=0}^{\eta_n} p^n_j\\
    	    					
  \end{cases}
  \]
  By shadowing, there is a point $z\in X$ with an orbit $(z_j)_{j=0}^\infty$ that $\epsilon$-shadows the pseudo-orbit $(w_j)_{j=0}^\infty$, and clearly $(z_j)_{j=0}^\infty$ is the orbit such that $d(z_j, x^i_j)< \epsilon$, for $1 \leq i \leq n$, $a_i \leq j \leq b_i$.
\end{proof}
\section{Inverse Limits}
There have been many publications on inverse limits of set-valued relations, but the associated dynamics of the bonding map and how the two relate have not been studied in depth. Here we study some dynamical properties that arise in the inverse limit setting.

\begin{definition}
Let $F\colon X \rightarrow 2^X$ be a set-valued map. Let $\sigma\colon \invlim F\rightarrow\invlim F$ be the forgetful shift map, so for $x =(x_1,x_2,\ldots)$,  $\sigma(x) = (x_2,x_3,\ldots)$.
\end{definition}
\begin{theorem}Let $X$ be a compact metric, and $F\colon X\rightarrow 2^X$ be a set-valued map.  Let $F$ have specification on $X$.  Then $\invlim F$ has specification via $\sigma$.
\end{theorem}
\begin{proof}
Let $\epsilon > 0$.  Let $k\in \N$ such that $$\sum_{i=k}^\infty \frac{1}{2^i} < \frac{\epsilon}{2}.$$ Let $M'$ witness specification for $\frac{\epsilon}{2}$. Let $M = M'+k$. Let $x^1,x^2,\ldots, x^n \in \invlim F$.  Let $1=a_1\leq b_1 < a_2 \leq b_2 \leq\ldots\leq b_n$ with $a_{i+1}- b_{i} > M$. Note that $x^i = (x^i_1,x^i_2,\ldots)$, with $x^i_j\in F(x^i_{j+i})$.  Consider, for each $i$, the finite piece of the orbit $(x^i_j)_{j=a_i}^{b_i+k}$.\\
We construct a point in the inverse limit that that traces our orbits.


Let  $\alpha_1 = 0$, $ \beta_1 = b_n+k-a_n$, and for $2\leq i \leq n$, $\alpha_i = b_n-b_{n-(i-1)}$, and $\beta_i = b_n+k-a_{n-(i-1)}$.  Note that $\alpha_{i+1}-\beta_i > M'$. Let $y_1 = x^n_{b_n+k}$, $y_2 = x^{n-1}_{b_{n-1}+k}$,$\ldots$ such that $y_i = x^{n-(i-1)}_{b_n+k}$ for each $1\leq i \leq n$. So by the specification property, there exists a point $z'$ in $X$,  that has a periodic orbit of length $D>M+b_n-a_1$, denoted $(z'_j)_{j=0}^\infty$ such that
\[|(z'_j)-(x^{n-(i-1)}_{b_n+k-j})| < \frac{\epsilon}{2},\;\;\;\text{for}\;\;\; 1\leq i \leq n,\;\;\alpha_i\leq j \leq \beta_i.\]
Let $z$ be the point in $\invlim F$ whose first coordinate is $z'_{b_n+k}$, second coordinate is $z'_{b_n+k-1}$, in general whose $i$th coordinate is $z'_{b_n+k-(i-1)}$, where we trace out the known periodic orbit of $z'$ for negative indices.
 We show that $z$ witnesses specification for $x^1,\ldots, x^n$, and the $a_i$'s, $b_i$'s.
   To see this, let $a_i \leq a \leq b_i$.

\begin{align*}
d(\sigma^{a}(z), \sigma^a(x^i)) &= \sum_{j=1}^{k-1}\frac{|\pi_j(\sigma^a(z))-\pi_j(\sigma^a(x^i))|}{2^j}+\sum_{j=k}^\infty \frac{|\pi_j(\sigma^a(z))-\pi_j(\sigma^a(x^i))|}{2^j}\\
&\leq \sum_{j=1}^{k-1}\frac{|(z'_{b_n+k-a-j})-(x^i_{a+j})|}{2^j}+\sum_{j=k}^\infty\frac{1}{2^j}\\
&\leq \sum_{j=1}^{k-1}\frac{|(z'_{b_n+k-a-j})-(x^i_{a+j})|}{2^j}+\frac{\epsilon}{2}.
\end{align*}
Now, note that as $a_i \leq a \leq  b_i$, we have  \[b_n+k-b_i-(k-1)\leq b_n+k-a-j \leq b_n+k-a_i-1\]
\[\alpha_{n-(i-1)} \leq b_n-b_i+1 \leq b_n+k-a-i \leq b_n+k-a_i-1 \leq  \beta_{n-(i-1)}.\]
Then $|(z'_{b_2+k-a-j})-(x^i_{a+j})|< \frac{\epsilon}{2}$, for each $j$.  Thus our sum becomes
\[\leq \sum_1^{k-1}\frac{\frac{\epsilon}{2}}{2^j}+\frac{\epsilon}{2} < \epsilon.\]
As $\sigma^D(z) = z$, we have that $\invlim F$ has specification via $\sigma$.
\end{proof}

For the next theorem, we need to define the inverse of a set-valued map on $X$.
\begin{definition} Let $F\colon X\rightarrow 2^X$ be a set-valued map.  Then $F^{-1}\colon X \rightarrow 2^X$ is the set-valued map on $X$ such that $F^{-1}(x) = \{y\in X:x\in F(y)\}$.
\end{definition}
\begin{lemma} Let $F\colon X\rightarrow 2^X$. The complete orbit of $F$ and the inverse limit of $F^{-1}$ are equal as sets.
\end{lemma}
%
This tells us that chaotic properties of $F$ will be reflected in the structure of the inverse limit of $F^{-1}$, and vice-versa.  The following theorem connects the ideas of the specification property and inverse limits.
\begin{theorem} Let $F\colon X\rightarrow 2^X$ be a set-valued map.
 $F$ has specification if, and only if, $F^{-1}$ does as well.
\end{theorem}
\begin{proof}
Suppose $F$ has specification.  Let $\epsilon >0$.  Let $M$ witness specification for $F$ and this $\epsilon$. Let $x^1,\ldots, x^n\in X$.  Let $a_1\leq b_i < \ldots \leq b_n$ with $a_{i+1}- b_{i} > M$.  Let $P > M+b_n-a_1$.  Let $(x^i_j)_{j=0}^\infty$ be an orbit of $x^i$ via $F^{-1}$, for each $1\leq i \leq n$.  Now consider the orbit segments $(x^i_j)_{j=a_i}^{b_i}$.  We have that $x^i_{j+1}\in F^{-1}(x^i_{j})$, and so $x^i_j\in F(x^i_{j+1}) $.  Let \[(y^i_j)_{j=a_i}^{b_i}=(x^{n-(i-1)}_j)_{j=b_{n-(i-1)}}^{a_{n-(i-1)}}.\]
By specification for $F$, there is some $z$ with an orbit $(z_j)_{j=0}^\infty$ such that
\[d(z_j, y^i_j)< \epsilon,\;\;\;\; 1\leq i \leq n,\;\;a_i \leq j \leq b_i.\]  The result follows from the proof of the previous theorem.  The converse follows similarly.
\end{proof}
This gives a way to determine a class of relations that have the specification property.  To be more precise, the inverses of continuous single-valued functions that have specification will have specification. As an example, the inverse of the tent map will have specification, see Figure 1.

\begin{figure}[h]
\begin{center}
\begin{tikzpicture}
\draw[-] (0,0) -- (0,0) node[left] {$(0,0)$};
\draw[-] (0,0) -- (5,0) node[right] {$x$} node[below] {$(0,1)$};
\draw[-] (0,0) -- (0,5) node[above] {$y$} node[left] {$(1,0)$} ;
\draw[-] (0,5) -- (5,5) ;
\draw[-] (5,0) -- (5,5) ;
\draw[-](0,0) -- (5,2.5) node[right] {$(1,\frac{1}{2})$};
\draw[-](0,5) -- (5,2.5);
\end{tikzpicture}
\caption{}
\end{center}
\end{figure}

\noindent  This implies that $\sigma$ on the inverse limit of a map of the interval with the specification also has the specification property.  Hence such $\sigma$'s will also have non-atomic invariant measures of full support.  The existence of such measures was the focus of \cite{kennedyrainesstockman}.

Earlier, we defined the forgetful shift map $\sigma\colon  \invlim F \rightarrow \invlim F$, which is a single-valued continuous function.  There is also an induced set-valued mapping on the inverse limit which we call $\gamma$.
\begin{definition}
Let $F\colon X\rightarrow 2^X$. Let $\gamma\colon  \invlim F \rightarrow 2^{\invlim F}$ be a set-valued map on $\invlim F$ such that $\gamma((x_0,x_1,\ldots)) = \{(x_{-1},x_0,x_1,\ldots): x_{-1} \in F(x_0)\}$.
\end{definition}
\begin{corollary}
Let $F$ be a set-valued map on $X$, a compact metric that has specification.  Then $\gamma$ has specification.
\end{corollary}
\begin{proof}
Identify with the shift map $\sigma$ a set-valued map $\sigma\colon \invlim F \rightarrow 2^{\invlim F}$ such that $\sigma((x_0,x_1,\ldots)) = \{(x_1,x_2,\ldots)\}$.  Then $\sigma$ and $\gamma$ are inverses as defined above, and so the result holds by Theorem 9.
\end{proof}
\begin{lemma}
Let $F\colon X \rightarrow 2^X$ be a set-valued map. If $F$ has mixing, then $F^{-1}$ has mixing.
\end{lemma}
\begin{proof}
Let $U$ and $V$ be non-empty open sets, we wish to show that there exists $M$ such that for $m > M$, there exists a point in $U$ that has an orbit $(x_j)_{j=0}^\infty$ under $F^{-1}$ such that $x_m \in V$.  As $F$ has mixing, let $N\in \N$ such that for all $n > N$, there exists $y^n\in V$ such that $y^n$ has an orbit $(y^n_j)_{j=0}^\infty$ under $F$ with $y^n_n\in U$.  Then letting $y^n_n$ be our choice for a point in $U$, we see that under $F^{-1}$, this point has an orbit whose $n$th iterate lands in $V$.
\end{proof}
\begin{corollary}
If $F$ is the inverse, as defined above, of a continuous single-valued function mapping the interval to itself, then $F$ has mixing iff $F$ has specification.
\end{corollary}
\begin{proof}
It is sufficient to note Theorem 9, Lemma 10, and the exceptional theorem from Blokh \cite{Blokh}, which states that mixing and the specification property are equivalent on single-valued interval maps.
\end{proof}

\section{Measures on a set-valued map.}
The specification property yields many good results in the measure spaces of single-valued functions, most notably the fact that in the space of invariant measures, there is a dense $G_\delta$ set of non-atomic measures with full support \cite{Denker}. We give some preliminary results towards the existence of an invariant measure with full support which is non-atomic.
%

To begin adapting this to the set-valued case, we introduce the notion of an invariant measure on a set-valued map.  Aubin, Frankowska and Lasota \cite{aubin} gave the following notion of an invariant measure, which Akin and Miller \cite{AkinMiller} showed to be equivalent to many other notions of an invariant measure.

\begin{definition}
Let $F\colon X\rightarrow 2^X$ be a set-valued map.  Let $P(X)$ be the space of Borel probability measures on $X$.  Then a measure $\mu\in P(X)$ is said to be \emph{invariant} if
\[\mu(B) \leq \mu(F^{-1}(B)),\]
for all Borel sets $B$ of $X$.
\end{definition}

To see a trivial example of an $F$-invariant measure, suppose $F\colon X\rightarrow 2^X$ has a point $x$ with a periodic orbit $(x_j)_{j=0}^\infty$, with period $n$.  Then define the periodic measure $\delta_x$ by
\[\delta_x(B) =\frac{|\{j\in \{0,1,\ldots, n-1\}: x_j \in B\}|}{n},\]   for any Borel set $B$. To see that $\delta_x$ is invariant, let $B$ be a Borel set in $X$, and suppose that there are $k$ many distinct elements of $\{x_j : j\in \{0,1,\ldots, n-1\}\}$ in $B$.  Then $F^{-1}(B)$ will have at least $k$ many distinct elements of $\{x_j : j\in \{0,1,\ldots, n-1\}\}$, and so the measure of $F^{-1}(B)$ will be at least the measure of $B$.

\begin{theorem}
Let $F\colon X\rightarrow 2^X$ be a set-valued map. The set of $F$-invariant measures with support $X$ is either empty or a dense $G_\delta$ set in the space of $F$-invariant measures on $X$.
\end{theorem}
\begin{proof}
Let $\mu$ be a $F$-invariant measure with support $X$.  Let $U$ be an open nonempty subset of $X$, and so $\mu(U) > 0$.  Denote $D(U) = \{ \mu | \mu\text{ is $F$-invariant}, \mu(U) = 0\}$.  Then $D(U)$ is a closed collection of measures.  Also, $D(U)$ has no interior.  To see this, let $\nu \in D(U)$, $\epsilon > 0$ and consider the measure $\nu_\epsilon =(1-\epsilon)\nu + \epsilon\mu$.  $\nu_\epsilon(U) = \epsilon\mu(U) > 0$, and so $\nu_\epsilon$ is not in $D(U)$.  Letting $\epsilon$ go to zero, $\nu_\epsilon \rightarrow \nu$, and so for each point of $D(U)$, there is a sequence of measures outside of $D(U)$ approaching that point.  Thus $D(U)$ is nowhere dense.\\
Now as $X$ is a compact metric, it is second countable, so let $\mathcal{U} =  \{U_1,U_2,\ldots\}$ be a countable basis for $X$.  Then $\displaystyle{\bigcup_{i=1}^\infty D(U_i)}$, being a countable collection of nowhere dense sets, is of first category.  Thus the complement of this union is a dense $G_\delta$ set.  As any measure in the complement of this union must have full support, the proof is complete.
\end{proof}

We now construct an invariant measure with full support, which gives us the following theorem.

\begin{theorem}
Let $F\colon X\rightarrow 2^X$ be a set-valued map. If $X$ has a dense set of points with periodic orbits (note that $F$ having specification gives this), then a dense $G_\delta$ set of invariant measures have full support.
\end{theorem}
\begin{proof}
It suffices to build an invariant measure with full support.  Let $\{x^i:i\in\N\}$ be a dense set of points with periodic orbits $(x^i_j)_{j=0}^\infty$, and $x^i_{j_i} = x_0$.  Consider the measure
\[\mu = \sum_{i=1}^\infty \frac{1}{2^i}\delta_{x^i},\]
with $\delta_x$ as defined above. To see that $\mu$  is a measure, let $\{U_n\}_{n=1}^\infty$ be a countable collection of disjoint measurable sets. For each $n\in \N$, let
\[k^i_n = |\{j\in \{0,1,\ldots, j_i-1\}: x^i_j\in U_n\}|. \]
Then we have that
\[\mu(U_n) = \sum_{i=1}^\infty \frac{1}{2^i}\delta_{x^i}(U_n) = \sum_{i=1}^\infty \frac{1}{2^i}\frac{k^i_n}{j_i},\text{ and so } \sum_{n=1}^\infty \mu(U_n) = \sum_{n=1}^\infty \sum_{i=1}^\infty\frac{1}{2^i}\frac{k^i_n}{j_i}.\]

Now we consider $\mu(\cup_{n=1}^\infty U_n)$.
\begin{align*}
\mu(\cup_{n=1}^\infty U_n) &= \sum_{i=1}^\infty \frac{1}{2^i}\delta_{x^i}(\cup_{n=1}^\infty U_n)\\
&=\sum_{i=1}^\infty \frac{1}{2^i}  \frac{|\{j\in \{0,1,\ldots, j_i-1\}: x^i_j \in \cup_{n=1}^\infty U_n\}|}{j_i}\\
&=\sum_{i=1}^\infty \frac{1}{2^i}  \frac{\sum_{n=1}^\infty|\{j\in \{0,1,\ldots, j_i-1\}: x^i_j \in  U_n\}|}{j_i}\\
&=\sum_{i=1}^\infty \sum_{n=1}^\infty \frac{1}{2^i}  \frac{k^i_n}{j_i.}\\
&=\sum_{n=1}^\infty \sum_{i=1}^\infty \frac{1}{2^i}  \frac{k^i_n}{j_i.}\\
&=\sum_{n=1}^\infty \mu(U_n).
\end{align*}
We now show that $\mu$ is invariant.  Let $B$ be a borel set, let $m_i = |\{j\in \{0,1,\ldots, j_i-1\} : x^i_j \in B\}|$, and let $n_i = |\{j\in \{0,1,\ldots, j_{i}-1\} : x^i_j \in F^{-1}(B)\}|$. Then
\[\mu(B) = \sum_{i=1}^\infty \frac{m_i}{2^i j_i}.\]
\[\mu(F^{-1}(B)) = \sum_{i=1}^\infty \frac{n_i}{2^i j_i}.\]
For $\mu$ to be invariant, it suffices to show that $n_i \geq m_i$.  To see this, let $j\in \{0,\ldots, j_i-1\}$ such that $x^i_j \in B$.  Then as $x^i_{j-1} \in F^{-1}(B)$, we see that $n_i \geq m_i$.
\end{proof}

 \bibliography{mybib}
\bibliographystyle{plain}
\end{document}